\documentclass{amsart}[12pt] 
\usepackage{latexsym,amssymb,graphicx,amscd,amsmath,url}

\newtheorem{thm}{Theorem}
\newtheorem{lem}[thm]{Lemma}
\newtheorem{prop}[thm]{Proposition}
 
\newtheorem{de}[thm]{Definition} 
\newtheorem{rem}[thm]{Remark}

\newtheorem{ques}[thm]{Question}

\newcommand{\BZ}{{\mathbb{Z}}}
\newcommand{\BQ}{{\mathbb{Q}}}
\newcommand{\BR}{{\mathbb{R}}}

\newcommand{\BE}{{\mathcal{E}}}

\newcommand{\BB}{{\mathcal{B}}}

\newcommand{\BC}{{\mathbb{C}}}

\newcommand{\DD}{{\mathcal{D}}}
\newcommand{\LL}{{\mathcal{L}}}

\newcommand{\Si}{{\Sigma}}

\newcommand{\fs}{{\mathfrak{s}}}

\newcommand{\NN}{{\mathcal{N}}}
\newcommand{\UU}{{\mathcal{U}}}

\newcommand{\al}{\alpha}
\newcommand{\g}{\gamma}

\newcommand{\I}{{\mathrm I}}

\DeclareMathOperator{\Span}{Span}

\DeclareMathOperator{\Hom}{Hom}

\DeclareMathOperator{\pr}{proj}

\DeclareMathOperator{\tr}{tr}
\DeclareMathOperator{\SU}{SU}
\DeclareMathOperator{\SO}{SO}

\DeclareMathOperator{\res}{res}
\DeclareMathOperator{\ev}{ev}

\begin{document}

\title[On the skein module of $\Sigma_g \times S^1$]{On the skein module of the product of a surface and a circle}
 
\author{ Patrick M. Gilmer}
\address{Department of Mathematics\\
Louisiana State University\\
Baton Rouge, LA 70803\\
USA}
\email{patgilmer@gmail.com}
\urladdr{\url{www.math.lsu.edu/~gilmer}}

\author{Gregor Masbaum}
\address{CNRS, Sorbonne Universit{\'e}, Universit{\'e} Paris Diderot, Institut de Math{\'e}matiques de Jussieu-Paris Rive Gauche, IMJ-PRG,
Case 247, 4 pl. Jussieu,
75252 Paris Cedex 5\\
FRANCE }
\email{gregor.masbaum@imj-prg.fr}
\urladdr{\url{webusers.imj-prg.fr/~gregor.masbaum/}}

\begin{abstract}  Let $\Sigma$ be a closed oriented surface of genus $g$. We show that the
Kauffman bracket skein module of $\Sigma \times S^1$  over the field of rational functions in $A$ has 
dimension at least $2^{2g+1}+2g-1.$
\end{abstract}



\maketitle\section{Introduction}\label{s1}
The Kauffman bracket \cite{K} skein module of a $3$-manifold was defined independently  by Przytycki \cite{P} and Turaev \cite{T1}. It can be defined over any commutative ring containing an invertible element usually denoted $A$. In this paper, we only consider this skein module over $Q(A)$, the field of rational functions in
a variable $A$. Given a connected oriented 3-manifold $M$, the skein module $K(M)$ is the vector space over $Q(A)$ generated by the set of isotopy classes of framed unoriented links in the interior of $M$ (including the empty link $\emptyset$ ) modulo the subspace generated by the Kauffman relations \cite{K}:

$$\begin{minipage}{0.4in}\includegraphics[width=0.4in]{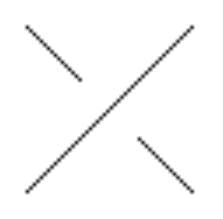}\end{minipage}= A \begin{minipage}{0.4in}\includegraphics[width=0.4in]{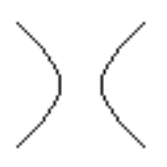}\end{minipage} + A^{-1} \begin{minipage}{0.4in}\includegraphics[width=0.4in]{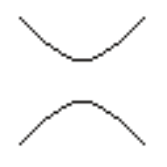}\end{minipage}$$
	$$L \cup \hbox{unknot} = (-A^{2} - A^{-2}) L$$

In the first relation, the  intersection of three framed links (with the blackboard framing) within a 3-ball in $M$ is depicted and  in the exterior of this 3-ball, all 3 links should be completed identically.
In the second relation by an unknot, we mean a loop which bounds a disk in the complement of $L$ with  the framing on the unknot extending to a non-zero normal vector field on the disk.  
Employing Kauffman's argument \cite{K} that justifies his bracket approach to the Jones polynomial,   one sees  that $K(S^3)$ is one dimensional and is generated by the class of  the empty link.
Given $\fs \in K(S^3)$, let $\left<  \fs  \right> $ denote its Kauffman bracket, which is defined to be the scalar   $\left<  \fs  \right>  \in Q(A)$ for which    $\left<  \fs  \right>  \emptyset $ and $\fs$ represent the same element in $K(S^3)$. 

Przytycki   \cite{P3} showed that
for connected sums of $3$-manifolds one has 
$K(M \# N)=K(M) \otimes K(N)$.
So the problem of computing $K(M)$ for compact $M$ reduces to computing $K(M)$ for prime manifolds. The complete answer is known for these prime 3-manifolds without boundary: 
 for the lens spaces and  
 for $S^2\times S^1$ (Hoste and Przytycki \cite{HP1,HP}), for
 the quaternionic manifold (Gilmer and Harris \cite{GH}), 
 for an infinite family of prime prism manifolds  (Mroczkowski \cite{Mr}), 
 and for the $3$-torus (see  the history of this calculation below). 
For compact 3-manifolds with boundary $K(M)$ is known  for  $\mathrm{I}$-bundles over surfaces (Przytycki \cite{P2}), 
for the product of a two-holed disk and 
a 
circle (Mroczkowski and Dabkowski \cite{MK}), 
for the exteriors of 2-bridged knots (Le \cite{Le}),
 for the exteriors of torus knots (March{\'e} \cite{M})
and for exteriors of 2-bridged links (Le and Tran \cite{LT}). 
There has been a lot of work on the algebra structure on the skein module of a surface times an interval (sometimes with coefficients different from $Q(A)$), and also on the connection between skein modules with character varieties.  We would like to mention \cite{BW,FKL,Le2} for some recent papers on this.

In the present paper, we wish to study the skein module of $\Sigma_g \times S^1$, where $\Sigma_g$ is an oriented closed surface of genus $g$. As already mentioned, two cases are known. Hoste and Przytycki showed that $K(\Sigma_0 \times S^1)= K(S^2 \times S^1)$ is one dimensional and is generated by the class of  the empty link. That $K(S^2 \times S^1)$ is at most one dimensional can easily be seen using the Dirac string trick.
Quite recently \cite{C}, Carrega  showed $K(\Sigma_1 \times S^1) = K(S^1 \times S^1 \times S^1)$  is  generated by a certain set of nine skein classes. Then the first author \cite{G} showed that Carrega's nine generators are linearly independent and thus $\dim(K(\Sigma_1 \times S^1)) =9.$
 Carrega's proof uses the Frohman-Gelca `product to sum formula' \cite{FG} 
(see also Sallenave \cite{S}) 
for the product in the skein algebra  $K(\Sigma_1 \times\I)$.

Our main result is to generalize the linear independence aspect of these results.
\begin{thm}\label{main} 
$\dim(K(\Sigma_g \times S^1)) \ge  2^{2g+1}+2g-1$. 
\end{thm}

Note that in genus zero and one this inequality is an equality. We don't know whether it is an equality in higher genus. 

\section{The evaluation map}\label{sec-eval}

The main ingredient in the proof of Theorem~\ref{main} is the Witten-Reshetikhin-Turaev TQFT invariant $\left<M,L\right>_{A}$ of a pair consisting of a closed oriented 
   $3$-manifold $M$ 
and an unoriented framed link $L$ in $M$.  Here $A$ is a primitive $2p$-th root of unity where $p=2d+1$ 
is
an odd integer greater than or equal to 
three.
We will use the so-called $\SO(3)$-invariants and we will adopt the skein theory approach of Blanchet, Habegger, Masbaum and Vogel \cite{BHMV1, BHMV2} to these  invariants.  We recall the   definition of $\left<M,L\right>_{A}$ in \S \ref{s2}.  
 
Let us consider the family of these invariants
associated to a choice of $A$ from the set \[\UU =\{ e^{ \pi i s/{p} }| \text{$p$ odd, } p >1, (s, 2p)=1  \}. \]
For every  unoriented framed link $L$ in $M$, the assignment $A\mapsto \left<M,L\right>_{A}$ defines a complex-valued function on $\UU$. As observed in \cite[Thm 1.7(i)]{BHMV2}, this function only depends on the skein class of $L$.

Let $\BC^\UU_{a.e.}$ 
denote the set of complex valued functions defined almost everywhere on $\UU$ where we consider two functions  to be equal if they agree almost everywhere on $\UU$. 
Because the rational functions which occur as coefficients of skein elements and skein relations   can have only finitely many poles, the assignment $A\mapsto \left<M,L\right>_{A}$ extends to a map $$\ev : K(M) \longrightarrow \BC^\UU_{a.e.}$$ which we call the {\em evaluation map}. Moreover, this map is $Q(A)$-linear where we  make $\BC^\UU_{a.e.}$ a vector space over $Q(A)$ by setting $R f(\g)= R(\g) f(\g)$ where $R \in Q(A)$, $f \in  \BC^\UU_{a.e.}$, $\gamma \in \UU.$    We summarize the above mentioned properties in the following Theorem.

\begin{thm}[Evaluation map]\label{wrt}
There is linear map of $Q(A)$ vector spaces $ \ev : K(M) \rightarrow \BC^\UU_{a.e.}$ which sends  a skein class represented by the link $L$ in $M$
to the  function in $\BC^\UU_{a.e.}$ which sends   $A \in \UU$ to  $\left<M,L\right>_{A}$, for almost all elements $A$ of $\UU$.
 \end{thm}

 The idea of using the evaluation  map 
 to study linear independence in skein modules was already used in \cite{GH,G}.
 We remark that March{\'e} and Santharoubane's very interesting paper \cite{MS} contains strong results about an $\SU(2)$-version of the evaluation map. We shall make some more comments about the connection  between their work and ours in \S \ref{sec6}.

The following Theorem \ref{p} gives a new tool for getting information from the evaluation map. To state it, note that there is an algebra structure on $\BC^\UU_{a.e.}$ given by the pointwise multiplication of functions.
We can view the order of a root of unity as a function defined on  $\UU$. 
More precisely we consider
the function on $\UU$ which assigns $p$ to $e^{{s\pi i}/{p}}$ 
(where $p>1$ is odd and 
      $(s,2p)=1$).
We  denote this function by $p$. Then, of course, $p^j$ will assign $p^j$ to $e^{{s\pi i}/{p}}$. 

 In the computations that lead to Theorem~\ref{main}, the image of the evaluation map always lies in the subspace of  $\BC^\UU_{a.e.}$ spanned by such functions. Therefore the following result will be important for us.
 
 \begin{thm} \label{p}  The subset of  the $Q(A)$ vector space $\BC^\UU_{a.e.}$  given by $\{p^j |j \in \BZ \}$ is  linearly  independent. \end{thm}

We give the proof of this theorem  in \S \ref{sec-p}.

\section{Proof of Theorem~\ref{main}}\label{sec-outline}

We can now give the proof of Theorem~\ref{main} modulo some results that will be proved in later sections.

We use the fact that the  skein module $K(M)$ has a direct sum decomposition 
\begin{equation} K(M)=\oplus_{x \in H_1(M;\BZ/2)} K_x(M) \label{ds} \end{equation} 
into submodules  
$K_x(M)$ which are defined similarly to $K(M)$ except one considers only links representing the given
$\BZ/2$-homology class. This  
holds because the skein relations respect the corresponding decomposition
of the vector space over 
$Q(A)$
generated by the set of isotopy classes of framed unoriented links in $M$ (including the empty link) where an isotopy class  of an unoriented framed link lies in the summand indexed by the $\BZ/2$-homology class that the link represents. Naturally,   the empty link is said to represent the zero homology class.

We will denote by 
$(x,y)$
 the element in $H_1(\Sigma_g \times S^1; \BZ/2)$ which projects to $x \in H_1(\Sigma_g ; \BZ/2)$ and $y \in \{0,1\}=H_1( S^1; \BZ/2)$.
  The embedding of $D^2 \times S^1$ into a neighborhood of  $* \times S^1$ 
   in $\Sigma_g \times S^1$, induces
  maps $$i_0: K_0(D^2 \times S^1)\rightarrow K_{(0,0)}(\Sigma_g \times S^1)$$
 $$i_1: K_1(D^2 \times S^1)\rightarrow K_{(0,1)}(\Sigma_g \times S^1)$$ where we use the decomposition $$K( D^2 \times S^1) = K_0(D^2 \times S^1) \oplus K_1(D^2 \times S^1)$$ into subspaces indexed by $H_1(D^2 \times S^1; \BZ/2) = \BZ/2$.

   The strategy now is to compose $i_0$ and $i_1$ with the evaluation map and thereby get lower bounds for the dimensions of $K_{(0,0)}(\Sigma_g \times S^1)$ and $K_{(0,1)}(\Sigma_g \times S^1)$.   Precisely, in \S\ref{prfTHeo} we will show the following:

  \begin{thm}\label{eo} 
Let $g \geq 1$. The 
  image 
  $(\ev \circ\, i_0)(K_0(D^2 \times S^1))$ is the $g+1$-dimensional subspace of $\BC^\UU_{a.e.}$ with basis 
       $$ \{ p^{g-1},p^{g+1}, p^{g+3},\ldots, p^{3g-3} \} \cup \{ p^g \} .$$
 The image
   $(\ev \circ\, i_1)(K_1(D^2 \times S^1))$ is the  $g$-dimensional subspace of $\BC^\UU_{a.e.}$ with the same basis except for $p^g$ which is omitted.
\end{thm} 

Note that this theorem relies on Theorem~\ref{p}, which tells us that the above mentioned powers of $p$ are linearly independent over $Q(A)$.
 Thus we have that
\begin{equation} \label{00} \dim(K_{(0,0)}(\Sigma_g \times S^1 )) \ge g+1 \quad \text{and} \quad \dim(K_{(0,1)}(\Sigma_g \times S^1 )) \ge g.\end{equation}

Theorem~\ref{main} follows from this and the following Lemma. Let $\NN$ denote the 
$2^{2g+1}-2$ elements of $H_1(\Sigma_g \times S^1 ;\BZ/2)$ which are different from $(0,0)$ or $(0,1)$.
\begin{lem}\label{lem6} For all $x\in \NN$, we have $ \dim(K_x(\Sigma_g \times S^1 ))\ge 1.$
\end{lem}

The  computation recorded in Lemma \ref{flat} below, which we will perform in \S \ref{s2},  together with Theorem \ref{wrt} proves Lemma~\ref{lem6} in  the special case where $x \in\NN$ 
 is represented by a non-separating simple closed curve on 
$\Sigma_g \times 1 \subset \Sigma_g\times S^1$.

\begin{lem}\label{flat} Let $\gamma$ denote a non-separating simple closed curve on $\Sigma_g \times 1$ with a framing tangent to $\Sigma \times 1$.
One has
\[ \left<\Sigma_g\times S^1, \gamma \right>_A= \left( \frac {-p} {(A-A^{-1})^2}\right)^{g-1} .
\] In particular, we have $\ev(\gamma)\not= 0$
and hence Lemma~\ref{lem6} holds for $x=([\gamma], 0)$.
\end{lem}

Here  $[\gamma]$ denote the class $\gamma$ represents in $H_1(\Sigma_g ;\BZ/2)$.

The general case of Lemma~\ref{lem6} now follows from this special case and the following:

\begin{lem}\label{trans} For all $x, y \in  \NN$, we have $ K_x(\Sigma_g \times S^1 ) \simeq K_{y} (\Sigma_g \times S^1 ) $.
\end{lem}

\begin{proof}[Proof of Lemma~\ref{trans}]
We use the fact that whenever we have an orientation-preserving diffeomorphism $f$ of an oriented $3$-manifold $M$ sending a {\em mod} $2$ homology class $x$ to $y=f_\star(x)$, then  $K_x(M)$ and $K_y(M)$ are isomorphic, the isomorphism being induced by $f$ in the obvious way. 

Thus it is enough to show that any two $x,y\in\NN$ are related by an orientation-preserving diffeomorphism of $\Sigma_g \times S^1$. 
Now any $x\in\NN$ is of the form $(c,0)$ or $(c,1)$ for some $c\in  H_1(\Sigma_g; \BZ/2) \setminus \{0\}$. Any two such classes $c,c'$ are related by an orientation-preserving diffeomorphism of $\Sigma_g$, as any such $c$ can be represented by a non-separating simple closed curve $\gamma$ on $\Sigma_g$, and any two non-separating simple closed curves on $\Sigma_g$ are related by an orientation-preserving diffeomorphism \cite{FM}. Thus we can easily deal with the case when $x$ and $y$ are both of the form $(c,0)$, or both of the form $(c,1)$.  

To complete the proof, we only need an orientation preserving diffeomorphism $\Sigma_g \times S^1 \rightarrow \Sigma_g \times S^1$ sending $(c,0)$  
to $(c,1)$ for some fixed $c\in  H_1(\Sigma_g; \BZ/2) \setminus \{0\}$. It can be constructed as follows. Represent $c$ by some non-separating curve $\gamma$ and 
pick $ \Phi \in H^1(\Sigma_g;\BZ)= \Hom( H_1(\Sigma_g;\BZ),\BZ)$ which takes the value $\pm 1$ on the integral homology class of $\gamma.$
There is 
a
smooth map $\phi:\Sigma_g \rightarrow S^1$, with  $\Phi= \phi_* : H_1(\Sigma_g;\BZ)\rightarrow H_1(S^1;\BZ)$.
Consider the diffeomorphism of $\Sigma_g \times S^1$ which sends $(a,z)$ to 
      $(a,\phi(a)\cdot z) $, 
where the dot indicates multiplication in $S^1$. This diffeomorphism 
preserves orientation and 
sends
 $(c, 0)$
to $(c,1)$.
\end{proof}

  \section{Proof of Theorem \ref{p}}\label{sec-p}
  
  Instead of using $p^j$ to denote the function in $\BC^\UU$ which assigns $p^j$ to $e^{{s\pi i}/{p}}\in \UU$, we denote  this function by $f_j$, in this section.

 Clearly $f_j$ is non zero.
 We suppose 
 there is  a linear dependence relation among  two or more of the $f_j$ and find a contradiction. Multiply this relation by $f_{-j}$ where $j$ is the highest index for $f$  appearing in the relation, and obtain:
  
  \[ R_n(A) f_{-n} + R_{n-1}(A) f_{1-n} + \dots + R_{1}(A) f_{-1} + R_{0}(A)f_0=0,
\]  
 where the $R_j(A)$ are rational functions of $A$ and $R_n(A) \ne 0$, $R_0(A) \ne 0$, and $n > 0.$
 Clearing denominators, we can assume, without loss of generality, that the $R_j(A)$ are polynomials in $A$.       
 Consider the locus  $C$ in $\BC \times \BC$ with coordinates 
 $\{z,w\}$ given by 
  \[ R_n(z) w^{n} + R_{n-1}(z) w^{n-1} + \dots  + R_{1}(z) w +   R_{0}(z)=0.
\]  
One has that $C$ is a finite union of irreducible algebraic curves.
Thus  $C$ defines a multivalued function on $\BC$ which assigns to $z \in \BC$ as values the elements of $\pr^{-1}(z)$ where  $\pr:C \rightarrow \BC$ denotes the projection on the first factor.
We say $z \in \BC$ is good if $z$ has a neighborhood $U_z$ for which $\pr: \pr^{-1} U_z \rightarrow U_z$ is a finite-sheeted trivial covering space. There are only a finite number of points $z \in \BC$ which are not good. The  function on $U_z$ whose graph is given by a single sheet of this cover is holomorphic and  determines a multivalued 
function on  $\BC$ by analytic continuation \cite{A}. This multivalued function will be the multivalued function  
 defined  by one of the irreducible components of $C$.

 For $u=    
 e^{s \pi i/p} 
 \in \UU$, we have that $(u,1/p)$ lies on $C$.

 Pick $r$ odd, such that  
 $e^{\pi i/2r}$ 
 is good. By the Chinese remainder theorem, there is  an infinite  sequence of integers  
 $p \ge 3$
 with $p = 1 \pmod{r}$  
 and $p = 3 \pmod{4}.$  Let $s_p=({p-1})/{2r}$.
  Then $z_p=e^{(s_p \pi i)/p}$ is a primitive $2p$-th root of unity   and  $\lim_{p \rightarrow \infty}z_p= e^{\pi i/2r}$. Thus for $p$ large in this sequence, 
 $z_p\in U_{e^{\pi i/2r}}$. By the pigeonhole principle, for infinitely many of these $p$, $(z_p,1/p)$ must  lie on the same sheet.   
 This sheet is the graph of some holomorphic function $F$ defined on $U_{e^{\pi i/2r}}$. So there is a subsequence of these $p$'s,  with $z_p\in  U_{e^{\pi i/2r}}$, and $\lim_{p \rightarrow \infty}z_p= e^{\pi i/2r}$,  $F(z_p)= 1/p$, 
  and $F({e^\frac{\pi i}{2r}}) =\lim_{p \rightarrow \infty}F(z_p)= 0$. 
 
Without loss of generality, we may assume $U_{e^{\pi i/2r}}$ misses the nonpositive real axis $\BR^-$. Let $\log$ denote the branch of the log function
defined on $\BC-\BR^-$ which takes values with imaginary part  in $(-\pi,\pi)$.  Consider \[ G(z)=\frac {\pi i -2r \log(z)}{\pi i}.\] It is an analytic function defined on $U_{e^{\pi i/2r}}$ with
$G(z_p)= 1/p$, 
 and \[G({e^\frac{\pi i}{2r}}) =\lim_{p \rightarrow \infty}G(z_p)= 0.\]
 
 Thus $F$
 and
 $G$ are 
 holomorphic  functions defined on  $U_{e^{\pi i/2r}}$ which agree at an infinite set of points with a limit point, and agree at this limit point. It follows that they agree on $U_{e^{\pi i/2r}}$. So they should have the same analytic extension to a multivalued function. But
the extension of $F$  is finitely multivalued  whereas the extension of $G$ is not, as it involves the logarithm function. This is the contradiction we seek. \qed

\section{Surgery formula for the WRT invariant and proof of Lemma \ref{flat}}\label{s2}

The skein module $K(D^2 \times S^1)= K( \I \times \I \times S^1)$ has an algebra structure given by stacking. It can be described as the ring of polynomials in the variable 
 $z$ with 
 coefficients
in $Q(A)$, where $z$ is  $\frac 1 2 \times \frac 1 2 \times  S^1 $ framed by one of the $\I$ factors.
A useful basis $\{ e_i | i \ge 0 \}$ of $K(D^2 \times S^1)$ as a vector space is defined recursively by 
$e_0= \emptyset $, $e_1=z,$ and $e_{i+1}= z e_i -e_{i-1}$. The bracket evaluation of this after embedding 
$D^2 \times S^1$ 
in the standard way in $S^3$ is denoted $\left<e_i \right>$. Moreover $\left<e_i \right>= (-1)^i [i+1]$, where we let $[n]$ denote the quantum integer 
$(A^{2n}-A^{-2n} )/(A^2-A^{-2})$.  
If $\al$ is a framed simple closed curve in a 3-manifold, then  `$\al$ colored $i$' describes the result of replacing $\al$ by $e_i$ in the $D^2 \times S^1$ neighborhood of $\alpha$ specified by the framing of $\al$.

Here is a  skein definition of  the invariant $\left<M,L\right>_{A}$  normalized as in \cite{BHMV2}. Let $A\in\UU$. For general $3$-manifolds $M$ (but not for $M=\Sigma_g\times S^1$, as will be explained in Remark~\ref{rem10} below) the invariant depends on an additional choice, namely the choice of a square root $\DD$ of the number   $\DD^2=
 -p/(A^2-A^{-2})^2$. ($\DD$ was called $\eta^{-1}$ in \cite{BHMV2}.) 
 Recall our notation  $d=(p-1)/2$ 
and define $\Omega_{p}\in K(D^2 \times S^1)$ by 
\begin{equation}\label{omega-eq}
\Omega_{p}=  \sum_{i=0} ^{d-1} \left<e_i \right>  e_i~.
\end{equation} 
\begin{de}\label{def}
Let  $M$ be a closed connected oriented 3-manifold and  $L$  be a framed link in  $M$. 
 Choose a framed link $\LL$ in $S^3$  whose linking matrix has signature zero such that $S^3$ surgered along $\LL$ is $M$, and  choose $L' \subset S^3 \setminus \LL$  a framed link whose image after the surgery is framed isotopic to $L$  in $M$.  
For $A \in \UU$, we define
\[\left<M,L\right>_{A}= \DD^{-1-\sharp\LL} 
\left<  \LL(\Omega_p) \sqcup L' \right>
 ,\]
where $\sharp\LL$ is the number of components of $\LL$, 
the bracket $\left< \ \ \right>$ on the R.H.S. is the Kauffman bracket evaluated at $A\in\UU$ (normalized so that $\left< \emptyset\right>=1$), 
and   $\LL(\Omega_p)$ denotes the result of 
 replacing a neighborhood of each component of $\LL$ by a copy of $\Omega_{p}$, using the framings on the components of  $\LL$ to determine the identification of  each neighborhood with 
      $D^2 \times S^1$. 
 \end{de}

\begin{rem}\label{rem10} {\em With the above hypothesis on the signature of the linking matrix the number of components of $\LL$ is congruent to the first Betti number $\beta_1(M)$ modulo ${2}$. Thus if the first Betti number is odd (which is the case for $M=\Sigma_g\times S^1$)  the invariant only involves $\DD^2$ and so  does not depend on the choice of $\DD$.  Note that in this case the invariant lies in the cyclotomic field $\BQ(A) \subset \BC$. But if $\beta_1(M)$  is even, the invariant is $\DD$ times an element of $\BQ(A)$.  For the purpose of studying $K(M)$ through the evaluation map, one may therefore want to normalize the invariant differently in that case.
}\end{rem} 

\begin{proof}[Proof of Lemma \ref{flat}]
There is a well known surgery description of $\Sigma_g \times S^1$. It is given by a  zero framed link $\LL$ with $2g+1$ components with all pairwise linking numbers zero as well.  It is the connected sum of $g$ copies of the borromean rings \cite[p.192]{Li}. See Figure~\ref{fig1} below where we include  $\gamma'$ as well.  According to 
      Definition~\ref{def}, 
$\left<\Sigma_g \times S^1, \gamma \right>_{A}$  is  obtained by multiplying 
$\DD^{- 2g-2}$ times the  bracket evaluation of the skein obtained by decorating  all $2g+1$ components with $\Omega_p$ and coloring $\gamma'$ with  $1$.  We can and do trade this color $1$ with $p-3$ without changing the bracket at a $2p$-th root of unity by  \cite[Lemma 6.3 (iii)]{BHMV1}.   
      Next, using the expression (\ref{omega-eq}) for $\Omega_p$, 
the long component in the figure which is decorated $\Omega_p$  is expanded as a sum  over 
$0 \le i \le d-1$ of $\left< e_i \right>$ times the diagrams  obtained   by recoloring this long component $i$.
 A fusion is performed to the three parallel strands which are encircled by the leftmost component colored $\Omega_p$. 
The three strands are colored $i$, $i$ and $p-3$ and the range of values
for $i$ is given by $0 \le i \le d-1$.  By a well-known  skein-theoretic argument going back to Roberts \cite[Fig. 7]{Ro}, only the colorings where $i$, $i$ and $p-3$ are $p$-admissible survive. This means only the term $i=d-1$ survives. (See {\em e.g.} \cite[\S 3]{GM1} for the $p$-admissibility conditions.)  Doing similar fusions on the  $2g-1$  pairs of strands now colored $d-1$ which are each encircled by a curve colored $\Omega_p$ leads to  $$\left<\Sigma_g \times S^1, \gamma \right>_{A}= \Bigg( \frac {\DD^2}{\left<e_{d-1} \right>^2}\Bigg)^{g-1}.$$ This can be rewritten as given in the lemma.
\end{proof}
 
\begin{figure}[h]
\centerline{\includegraphics[width=4.5in]{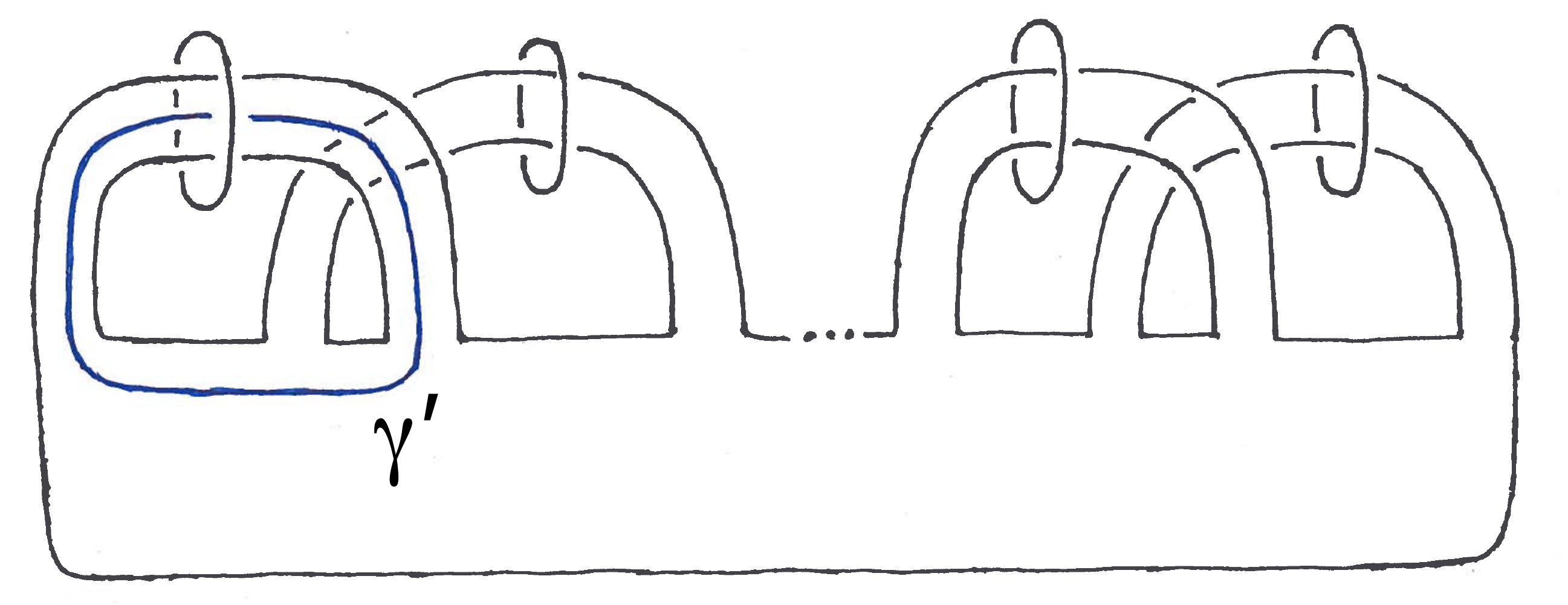}}
\caption{Surgery description of $(\Sigma_g \times S^1, \gamma)$}\label{fig1}
\end{figure}

 \section{ Verlinde formula and the proof of Theorem \ref{eo}}\label{prfTHeo}

The Verlinde formula gives the dimension of the TQFT vector space $V_p(\Si_g,2c)$ assigned to $\Si_g,$ with a point colored $2c$, where  $A$ is a
 primitive $2p$-th root of unity.
We denote this dimension by $D_g^{(2c)}$ with $p$ absent from the notation but understood. We will use the fact that $D_g^{(2c)}$ can be written as a polynomial $p$ and $c$. We will give more details about this polynomial in \S \ref{sec7} where the following two lemmas will be proved.  Let 
  \[ \BE= \{ g-1,g+1, \ldots, 3g-3 \}= \{g-1+2j \, \vert \,  j=0,1, \ldots, g-1\} .
   \]

\begin{lem}\label{p^g} Let  
      $g\geq 1$.
We have $$D_g^{(2c)}=\sum_{j\in\BE \cup \{g\}} \varphi_j(c)\, p^j~,$$ where $\varphi_j(c)$ is a polynomial in $c$ of degree exactly $3g-2-j$. 
\end{lem}

\begin{lem}\label{p^go} Let  
       $g\geq 1$.
  We have $$D_g^{(p-2s-1)}=\sum_{j\in\BE} \widetilde\varphi_j(s)\, p^j~,$$ where $\widetilde\varphi_j(s)$ is a polynomial in $s$ of degree exactly $3g-2-j$. 
\end{lem}

Here we say that a polynomial $f(x)$ has degree exactly $n$ if it has degree $n$ and the coefficient of $x^n$ is non-zero. 
In genus one, 
one has 
$D_1^{(2c)}= \frac p 2 - c - \frac 1 2$ 
and (hence) $D_1^{(p-2s-1)}=s$.

\begin{proof}[Proof of  Theorem \ref{eo}]
Let $\ell_i$ be the skein in $\Si_g \times S^1$ given by placing $e_i$ along 
$* \times S^1$. For simplicity of notation, let us write $\left<\ell_{2c} \right>_A$ for $\left<\Sigma_g \times S^1, \ell_{2c} \right>_A$. 
By the trace property of TQFT, we have  $$\left<\ell_{2c} \right>_A =\dim V_p(\Si_g,2c)= D_g^{(2c)}~.$$
Since $K_0(D^2 \times S^1)$ is spanned by $\{ e_{2c}\,|\, c \geq 0 \}$, the image of 
      $i_0: K_0(D^2 \times S^1)\rightarrow     K_{(0,0)}(\Sigma_g \times S^1)$                           
is spanned by
$\{ \ell_{2c}\,|\, c \geq 0 \}$. Writing $W_0$ for $(\ev \circ\, i_0)(K_0(D^2 \times S^1))$, we have that $W_0$ is the subspace of $\BC^\UU_{a.e.}$ spanned by  $\{\left<\ell_{2c} \right>_A |\, c\geq 0 \} =  \{D_g^{(2c)} \, | \,c\geq 0 \}$. By Lemma~\ref{p^g}, each  $D_g^{(2c)}$ is expressed as a linear combination of powers of $p$. More precisely, each  $D_g^{(2c)}$ lies in the span of $$\BB=\{p^j \, \vert \, j\in\BE\cup\{g\} \} \subset \BC^\UU_{a.e.}~.$$

Now recall from Theorem~\ref{p} that $\BB$ is  a linearly  independent set in the $Q(A)$ vector space $\BC^\UU_{a.e.}$. We must show that $W_0$ is precisely the subspace with basis $\BB$. We already know that  $W_0$ is contained in this subspace. The reverse inclusion is proved as follows. The coordinate vector of $D_g^{(2c)}$ with respect to $\BB$ is the $(c+1)$-th column of a  $(g+1) \times  \infty $ matrix whose $j$-th row is given by the values $\varphi_j(c)$ ($c\geq 0$). Again by Lemma~\ref{p^g}, the $g+1$ polynomials $\varphi_j$ have distinct degrees and so are linearly independent. This implies that the  $g+1$ rows of our matrix are linearly independent. (Indeed, any non-trivial linear relation between the rows $$\sum_j \lambda_j \, (\varphi_j(0), \varphi_j(1), \ldots, \varphi_j(c),  \ldots ) =0$$ would imply that $\sum_j \lambda_j \varphi_j $ is a polynomial with infinitely many roots, hence zero, giving us a non-trivial linear relation between the $\varphi_j$.)  Thus our matrix has rank $g+1$ and so $W_0$ is indeed the subspace with basis $\BB$. This completes the proof of  the first statement  in Theorem~\ref{eo}.

 For the second statement in Theorem~\ref{eo}, let us  write $W_1$ for $(\ev \circ\, i_1)(K_1(D^2 \times S^1))$. Since $K_1(D^2 \times S^1)$ is spanned by $\{ e_{2s-1}\,|\, s \geq 1 \}$, we have, as before,  that $W_1$ is the subspace of $\BC^\UU_{a.e.}$ spanned by  $\{\left<\ell_{2s-1} \right>_A | s\geq 1 \}$. We now use a property of $\SO(3)$-TQFT: one can recolor a link component of a  link $L$ colored $2s-1$ with the color $p-2s-1$ without changing $\left<L\right>_{A}$ if $A$ is a primitive $2p$-th  root of unity where $p \ge 2s+1$ (see \cite[Lemma 6.3 (iii)]{BHMV1}). Thus 
for almost all $A$, 
\[   \left<\ell_{2s-1} \right>_A= D_g^{(p-2s-1)}.\] 
 Using now Lemma \ref{p^go} in place of  Lemma \ref{p^g} and arguing as above, it follows that $W_1$ is the subspace of $\BC^\UU_{a.e.}$ with basis $\{p^j \, \vert \, j\in\BE \}$. This completes the proof of Theorem \ref{eo}.
 \end{proof}

Note that for the recoloring argument used in the above proof, it is important that we use $\SO(3)$-TQFT and not $\SU(2)$-TQFT.
\section{Proof of Lemmas~\ref{p^g} and~\ref{p^go}}\label{sec7}

We begin by collecting some properties of Bernoulli numbers and Bernoulli polynomials that will be needed. The Bernoulli numbers $B_k$ are defined by 
\begin{equation*} \frac {t} {e^{t}-1}= \sum_{k=0}^\infty B_k
\frac {t^k} {k!}~.
\end{equation*} We have $B_0=1$, $B_1=-1/2$ and $B_n=0$ for all odd $n\geq 3,$ while $B_{2n} \ne 0$ for all $n\geq 0$. The Bernoulli polynomials $B_m(x)$ can be defined by either of the following two expressions which can be found in  \cite[Ex. 12, p.248]{IR}: 
\begin{equation}\label{bern1}
B_m(x)=\sum_{\ell =0}^m \binom m \ell x^{m-\ell} B_\ell~.
\end{equation} 
\begin{equation}\label{bern2}
\sum_{n=0}^\infty B_n(x) \frac {t^n} {n!} = \frac {t e^{tx}} {e^t-1}~.
\end{equation}
Finally we will need the fact that
\begin{equation}\label{bern3}
B_{m}(\textstyle{\frac 1 2} )=(2^{1-m}-1) B_m .
\end{equation} (see \cite[Ex.22 p.248 with typo corrected]{IR}).

We can now give the proof of Lemmas~\ref{p^g} and~\ref{p^go}.
We start with the following residue formula for $D_g^{(2c)}$.  We use the notation
\begin{equation*}
\mathrm{s}(t)=
\frac{\sinh(t)}{t}= \sum_{k=0}^\infty \frac {t^{2k}}{(2k+1)!}~.
\end{equation*}
Then for 
   $g\geq 1$,
we have 
\begin{equation}\label{f5}
D_g^{(2c)} = \frac{(-p)^g}{2} 
\left( 4^{1-g}\frac{2c+1}{p}   \res_{t=0}\left(\frac{2pt}{e^{2pt}-1} 
\frac{\mathrm{s}\left((2c+1)t\right)}{\mathrm{s}(t)^{2g-1}} \frac{dt}{t^{2g-1}}\right)
- \binom{c+g-1}{2g-2}
\right)
\end{equation}
 A proof of this residue formula is given in \cite[Prop. 6.6]{GM}. Recall that for a Laurent series in a variable $t$, its residue at $t=0$ is simply the coefficient of $t^{-1}$. Thus the residue formula gives an explicit expression for $D_g^{(2c)}$  as a polynomial in $p$ and $c$. Moreover, it not hard to see that this polynomial is of the form 
\begin{equation}\label{f6}
D_g^{(2c)}=p^{g-1}X(p,c) + p^g Y(c)
\end{equation} for polynomials $X(p,c)$ and $Y(c)$ where $X(p,c)$ is even as a polynomial in $p$. (Here $X(p,c)$ comes from the residue term in (\ref{f5}), and $Y(c)$ comes from the binomial coefficient in (\ref{f5}). The fact that $X(p,c)$ is even as a polynomial in $p$ follows from the fact that $\mathrm{s}(t)$ is an even power series in $t$.) The following proposition gives the leading term of this polynomial. 
 
\begin{prop}\label{leadt} Let $F_n(p,c)$ be the homogeneous part of total degree $n$ in $D_g^{(2c)}$ viewed as a polynomial in $p$ and $c$. One has $F_n(p,c)=0$ for $n>3g-2$ and 
 \begin{equation}\label{f7} F_{3g-2}(p,c)= (-1)^g p^{g-1}\sum_{k=0}^{2g-1} \frac{B_k}{k!(2g-1-k)!}
  {c^{2g-1-k}}p^{k}
\end{equation}
\end{prop}

Proposition~\ref{leadt}  is proved from the residue formula (\ref{f5}) by a straightforward computation which we leave to the reader. 

\begin{proof}[Proof of Lemma \ref{p^g}] Since $D_g^{(2c)}$ has total degree $3g-2$ in $p$ and $c$, it can be written as $\sum_j\varphi_j(c)\, p^j~$ for some polynomials $\varphi_j$ in $c$ of degree $3g-2-j$. But since $D_g^{(2c)}$ is also of the form given in Eq.~(\ref{f6}), the polynomials $\varphi_j$ are zero except for $j\in\BE \cup \{g\}$ where $\BE= \{g-1+2j \, \vert \,  j=0,1, \ldots, g-1\}$ as defined in the previous section. It only remains to show that the $g+1$ non-zero polynomials $\varphi_j$ have degree exactly $3g-2-j$. But this is clear from Proposition~\ref{leadt}, which tells us that, up to multiplication by a non-zero number, the leading term of $\varphi_{g-1+2j}$ is  the Bernoulli number $B_{2j}$, and the leading term of $\varphi_g$ is the Bernoulli number $B_1$. Since these Bernoulli numbers are non-zero, this completes the proof. 
\end{proof}

For the proof of Lemma \ref{p^go}, observe that we obtain $D_g^{(p-2s-1)}$ from $D_g^{(2c)}$ by the substitution $c=(p-1)/2 - s$. Thus  $D_g^{(p-2s-1)}$ is again a polynomial of total degree $3g-2$ in $p$ and $s$. We need the following lemma whose proof we delay.

\begin{lem}\label{bernproof} $D_g^{(p-2s-1)}$ can be written as $p^{g-1}$ times a polynomial in $p$ and $s$ which is even as a polynomial in $p$.
\end{lem} 
\begin{proof}[Proof of Lemma \ref{p^go}] By Lemma~\ref{bernproof} and formula (\ref{f6}), we see that  $D_g^{(p-2s-1)}$ (being of total degree $3g-2$ in $p$ and $s$) can be written as $\sum_{j\in\BE} \widetilde\varphi_j(s)\, p^j~$ where $\widetilde\varphi_j$ is a polynomial in $s$ of degree $3g-2-j$. We only have to show that for $j=g-1+2k\in\BE$, the leading coefficient of $\widetilde\varphi_j$ is non-zero. This coefficient is the coefficient of $p^{g-1+2k}s^{2g-1-2k}$ in $F_{3g-2}(p,p/2-s)$. A straightforward calculation using the expression (\ref{f7}) gives that this coefficient is $$ \frac {(-1)^{g+1}} {(2g-2k+1)! (2k)!} \sum_{\ell =0}^{2k} \binom{2k}{\ell} 2^{-2k+\ell} B_\ell~.$$ Up to the non-zero prefactor  $(-1)^{g+1}/((2g-2k+1)! (2k)!)$, this is $B_{2k}(\textstyle \frac 1 2)$ which is a non-zero multiple of $B_{2k}$ by Eq.~(\ref{bern3}) and hence is non-zero. 
\end{proof}

\begin{proof}[Proof of Lemma \ref{bernproof}] We proceed by induction on the genus $g$. It will be convenient to prove the following stronger
\vskip 3pt
Statement $\mathrm{S}_g$: {\em  $D_g^{(p-2s-1)}$ can be written as $p^{g-1}$ times a polynomial in $p$ and $s$ which is even as a polynomial in $p$ and odd as a polynomial in $s$.}
\vskip 3pt
Statement $\mathrm{S}_1$ holds since $D_1^{(p-2s-1)} =\dim V_p(\Sigma_1,p-2s-1) = s$.

For the induction step, we use the fusion rules of $\SO(3)$-TQFT. Let $$K_{s,y} = \dim V_p(\Sigma_1, (p-2s-1, p-2y-1))$$ be the dimension of the TQFT vector space associated to a genus one surface with two colored points colored $p-2s-1$ and $p-2y-1$ respectively. Here $1\leq s,y \leq d=(p-1)/2$. Then the fusion rules imply the following recursion formula: 
\begin{equation}\label{fus}
D_{g+1}^{(p-2s-1)}=\sum_{y=1}^d K_{s,y} D_g^{(p-2y-1)}~.
\end{equation}

The induction hypothesis  $\mathrm{S}_g$ states that monomials occurring in $D_g^{(p-2y-1)}$ are of the form $p^{g-1+2\alpha} y^{2\beta -1}$ for some integers $\alpha \geq 0$ and $\beta \ge 1$.  By the following Lemma~\ref{lem-last}, in the recursion formula (\ref{fus}) each of these monomials gives rise  to a sum of monomials occuring in $D_{g+1}^{(p-2s-1)}$ which satisfy 
      the conclusion of $\mathrm{S}_{g+1}$.
\end{proof}

\begin{lem}\label{lem-last} For all integers $\beta \geq 1$, we have that $\sum_{y=1}^d K_{s,y}y^{2\beta-1}$ can be written as $p$ times a polynomial in $p$ and $s$ which is even as a polynomial in $p$  and odd as a polynomial in $s$.
\end{lem}
\begin{proof}[Proof of Lemma \ref{lem-last}]
By \cite[Lemma 7.1]{GM}, we have $$ K_{s,y} = (p-2\max(s,y)) \min(s,y)~.$$
Splitting the sum into two parts, 
we get 
$$\sum_{y=1}^d K_{s,y}y^{2\beta-1} = \sum_{y=1}^s (p-2s)y^{2\beta} + \sum_{y=s+1}^d (p-2y)sy^{2\beta-1}~.$$
Collecting terms, this is equal to $pU(s) +sV(p)$ where

\begin{align*} U(s) &= \sum_{y=1}^s y^{2\beta} - s \sum_{y=1}^s y^{2\beta-1}\\
 V(p) &= p \sum_{y=1}^d y^{2\beta-1} -  2\sum_{y=1}^d y^{2\beta}
\end{align*}

Now recall the classical Faulhaber formula 
\begin{align}\label{expr1} \sum_{y=1}^N y^m &= \frac{N^m} 2 + \frac {N^{m+1}}{m+1} \sum_{j=0}^{\lfloor m/2 \rfloor} \binom {m+1} {2j}  B_{2j} N^{-2j}\\
\label{expr2} &= \frac 1 {m+1} (B_{m+1}(N+1) - B_{m+1})~.
\end{align}
(For the first expression, see \cite[Theorem 15.1]{IR} and recall that the odd Bernoulli numbers are zero except for $B_1=-\frac 1 2$. For the second expression, see \cite[p.231]{IR}.)  Using the expression (\ref{expr1}) with $N=s$, we see that $U(s)$ is an odd polynomial in $s$. Using the expression (\ref{expr2}) with $N=d=(p-1)/2$ and the following Lemma~\ref{lem-last2}, we see that $V(p)$ is an odd polynomial in $p$. Lemma~\ref{lem-last} follows from this.
\end{proof}

\begin{lem}\label{lem-last2}   We have that $B_{2\beta}((p+1)/2)$ is an even polynomial in $p$, and $B_{2\beta+1}((p+1)/2)$ is an odd polynomial in $p$.
\end{lem}
\begin{proof}[Proof of Lemma \ref{lem-last2}]

 We use the expression (\ref{bern2}) for the Bernoulli polynomial $B_n(x)$. Applying it first to $x=(p+1)/2$ and then to $x=1/2$, we have
\begin{equation}\label{com}\sum_{n=0}^\infty B_n({\textstyle\frac {p+1} 2}) \frac {t^n} {n!} = e^{tp/2} \frac {t e^{t/2}} {e^t-1} = e^{tp/2} \sum_{n=0}^\infty B_n({\textstyle \frac {1} 2}) \frac {t^n} {n!}~.\end{equation} 
Now observe that  the right hand side of (\ref{com}) is $e^{tp/2}$ times an even power series in $t$, as we have $B_n({\textstyle \frac {1} 2})= (2^{1-n}-1)B_n$ by Eq.~(\ref{bern3}) and this is zero for all odd $n$ (including $n=1$). The lemma now follows by comparing the coefficients of $t^{n}$ on both sides of (\ref{com}).
\end{proof}

\section{Further comments} \label{sec6}

\subsection{The question of finite dimensionality} While our Theorem~\ref{main} gives a lower bound for the dimension of $K(\Sigma_g\times S^1)$, no upper bound for this dimension is known at present when $g\geq 2$. In fact it is not even known whether $K(\Sigma_g\times S^1)$ is finite-dimensional for $g\geq 2$. 
We would like to mention that one may ask more generally whether  $K(M)$ is finite-dimensional for all closed manifolds $M$. We first heard about this question  in Oberwolfach in June 2015  from Julien March\'e, who had heard about it from Greg Kuperberg, who had been asked about it by Ed Witten.  We thank Greg Kuperberg for informing us that Witten had asked this question because of a $3+1$-dimensional TQFT which would associate an invariant in $K(\partial W)$ to a $4$-manifold $W$ with boundary $\partial W$.  
(Note that if such a TQFT exists and satisfies the usual finite-dimensionality axioms, then  $K(M)$ would indeed have to be finite-dimensional for all closed $3$-manifolds $M$.) 
 We were, however, unable to locate an explicit statement of this  question in Witten's papers.
The first written mention of this question seems to be in Carrega's paper \cite{C}.

No counterexamples are known. In particular,  for the prime closed 3-manifolds $M$ mentioned in the 
introduction where $K(M)$ is known, $K(M)$ is always finite dimensional. Moreover, work of Bullock  \cite{B} implies that  $K(M)$ is finite dimensional for all non-zero integral surgeries on the trefoil knot. Also work of Harris \cite{H} implies  that integral surgery on a $(2,2b)$ torus link with framings $a$ and $c$ will 
 produce $3$-manifolds $M$ with  finite dimensional $K(M)$
 if $a$, $b$ and $c$ satisfy a certain set of inequalities.

\subsection{On the image of the evaluation map} It is natural to ask whether  the inequality for $\dim K(\Sigma_g\times S^1)$ given in Theorem~\ref{main} is an equality. In view of what we have shown, this inequality is an equality if and only if the answer to both of the following questions is `yes':

\begin{ques} Is $\ev: K_x(\Sigma_g\times S^1) \rightarrow \BC^\UU_{a.e.}$ injective for every $x\in H_1(\Sigma_g\times S^1; \BZ/2)$?
\end{ques}
\begin{ques}\label{ques17} Is $\ev (K_x(\Sigma_g\times S^1)) \subset \BC^\UU_{a.e.}$ given by what we already found? In other words, is it true that
\begin{align}\ev (K_x(\Sigma_g\times S^1))\  \stackrel{\text{?}}{=} \  
\begin{cases} 
\Span_{Q(A)}\{p^j\, \vert \, j \in \BE \cup \{g\} \} & \text{\ if\ } x=(0,0) \\
\Span_{Q(A)}\{p^j\, \vert \, j \in \BE  \} & \text{\ if\ } x=(0, 1)\\
\Span_{Q(A)}\{p^{g-1} \} & \text{\ if\ } x\in\NN
\end{cases}
\end{align}
\end{ques}
\noindent (Recall $\BE=  \{g-1+2j \, \vert \,  j=0,1, \ldots, g-1\}$.) Note that we know that the R.H.S. is included in the L.H.S. by  Theorem~\ref{eo} and Lemmas~\ref{flat} and ~\ref{trans}. So the question is whether the L.H.S. is included in the R.H.S. 

At the time of this writing, we don't know the answer to these questions. We remark that in addition to the computations presented in the preceding sections of this paper, we have computed the evaluation map for skeins living in a `horizontal' solid torus,  {\em i.e.,} a tubular neighborhood of a simple closed curve $$\gamma \subset \Sigma_g\times 1 \subset \Sigma_g\times S^1~.$$ The result of these additional computations is compatible with a positive answer to Question~\ref{ques17}.

 Here is a more precise statement. Let $\gamma(m)$ denote $\gamma$ colored $m$, and assume w.l.o.g. that  $\gamma$ has framing tangent to $\Sigma_g\times 1$ as in Lemma~\ref{flat}.

\begin{prop}\label{prop18} If $\gamma$ is a non-separating simple closed curve on $\Sigma_g\times 1$, then 
\begin{equation}\label{nonsepa} \ev(\gamma(m)) \ = \ 
\begin{cases} 
D_g^{(0)} - \displaystyle{\sum_{i=1}^{m/2} \frac {(-p)^{g-1}} {(A^{2i} - A^{-2i})^{2g-2}}} & \text{\ if\ } m \equiv 0 \pmod 2 \\
\displaystyle{\sum_{i=1}^{(m+1)/2} \frac {(-p)^{g-1}} {(A^{2i-1} - A^{-2i-1})^{2g-2}}} & \text{\ if\ } m \equiv 1 \pmod 2 
\end{cases}
\end{equation} In particular, we have 
\begin{equation} \ev(\gamma(m)) \ \in \ 
\begin{cases}\Span_{Q(A)}\{ p^{g-1}, D_g^{(0)}  \} & \text{\ if\ } m \equiv 0 \pmod 2 \\ 
\Span_{Q(A)}\{p^{g-1} \}& \text{\ if\ } m \equiv 1 \pmod 2 
\end{cases}
\end{equation}

 If $\gamma$ is an essential separating curve, then
\begin{equation}\label{refsepa} \ev(\gamma(m)) \ \in \ 
\begin{cases}\Span_{Q(A)} \{p^j\, \vert \, j \in \BE \}  & \text{\ if\ } m \equiv 0 \pmod 2 \\ 
\Span_{Q(A)} \{p^j\, \vert \, j \in \BE\setminus \{3g-3\} \}& \text{\ if\ } m \equiv 1 \pmod 2 
\end{cases}
\end{equation} In particular,  $\ev(\gamma(m))$ lies in a codimension one subspace of  $(\ev \circ\, i_0)(K_0(D^2 \times S^1))$ as computed in Theorem~\ref{eo} (and, moreover, in a codimension two subspace when $m$ is odd). 
\end{prop}

We also have explicit formulae for $\ev(\gamma(m))$ when $\gamma$ is separating but we omit them as they are not as nice as Formula~(\ref{nonsepa}). Implementing these formulae on a computer we can see that at least experimentally all powers of $p$ allowed by (\ref{refsepa}) may indeed occur in   $\ev(\gamma)$ for separating curves $\gamma$.

Note that the special case of Proposition~\ref{prop18} where $\gamma$ is non-separating and $m=1$ is contained in Lemma~\ref{flat} which is proved in 
      \S\ref{s2}.
The remaining cases of Proposition~\ref{prop18} can be proved by employing some techniques from $\SO(3)$-TQFT, starting from the expression  $$\ev(\gamma(m)) =\Big(\left< \gamma(m)\right>_A\Big)_{A\in\UU}= \Big( \tr_{V_p(\Sigma_g)} Z_p(\Sigma_g\times\I, \gamma(m) \times {\textstyle\frac 1 2} )\Big)_{A\in\UU}~.$$ We omit the details, as they are somewhat involved. Note that these remaining cases of Proposition~\ref{prop18} are not needed in the proof of our main result.

 \subsection{Connections with the work of March{\'e} and Santharoubane} In their recent paper \cite{MS}, March{\'e} and Santharoubane use the $\SU(2)$-WRT invariants of links in  $\Sigma_g\times S^1$ to construct, for every framed link $L$ in $\Sigma_g\times S^1$,  a Laurent polynomial $P_L\in \BZ[A,A^{-1}]$ with the following remarkable property. Let $\left<\Sigma_g\times S^1,L\right>_{{\zeta_{4r}}}^{\SU(2)}$ be the $\SU(2)$-invariant at some $4r$-th root $\zeta_{4r}\in\BC$. Then for all but finitely many $z$ on the unit circle, for any sequence of $4r$-th roots $\zeta_{4r}$ converging to $z$ (with $r$ going to infinity), the associated sequence of $\SU(2)$-invariants when divided by the invariant of the empty link converges to  $P_L(z)$; in fact there is an asymptotic expansion  $$\frac  {\left<\Sigma_g\times S^1,L\right>_{\zeta_{4r}}^{\SU(2)}} {\left<\Sigma_g\times S^1,\emptyset\right>_{\zeta_{4r}}^{\SU(2)}} = P_L(z) + O({\textstyle\frac 1 r}), \ \ \ \ \ \ (\zeta_{4r} \rightarrow z, r\rightarrow \infty)~. $$ 
March{\'e} and Santharoubane also show that the assignment $L\mapsto P_L$ satisfies the Kauffman skein relations, and develop methods to compute $P_L$ explicitly. 

Just as in $\SO(3)$-TQFT, one has that $$\left<\Sigma_g\times S^1,\emptyset\right>_{\zeta_{4r}}^{\SU(2)}= \dim V_{2r}^{\SU(2)}(\Sigma_g)$$
is a polynomial in $r$ of degree $3g-3$ (for $g\ge 2$). 
Thus March{\'e} and Santharoubane's theorem says that almost everywhere on the unit circle the invariant $\left<\Sigma_g\times S^1,L\right>_{{\zeta_{4r}}}^{\SU(2)}$ (which we may think of as the $\SU(2)$ evaluation map) grows at most like $r^{3g-3}$. Moreover,  the coefficient  of the term of order $r^{3g-3}$ is given in the limit as $r\rightarrow\infty$ by a polynomial $P_L(z)$ (again almost everywhere on the unit circle).

Note that if Question~\ref{ques17} had an affirmative answer, then for every framed link $L$ in $\Sigma_g\times S^1$ we would have $$\ev(L)= \sum_{j=g-1}^{3g-3} R_L^{(j)}(A)\, p^j$$ for some rational functions $R_L^{(j)}(A)$ (with $R_L^{(j)}(A)=0$ when $j\notin \BE \cup \{g\}$) and the rational function  $R_L^{(3g-3)}(A)$ would play a role similar to March{\'e} and Santharoubane's polynomial $P_L(z)$.

 \end{document}